\documentclass{amsart}
\makeatletter
\let\@@pmod\pmod
\DeclareRobustCommand{\pmod}{\@ifstar\@pmods\@@pmod}
\def\@pmods#1{\mkern4mu({\operator@font mod}\mkern 6mu#1)}
\makeatother
\usepackage{amssymb,mathtools,fullpage,graphicx,url,xcolor}
\usepackage{booktabs,array,enumerate,verbatim}
\usepackage{tikz,pgfplots,pgfplotstable}
\newtheorem{theorem}{Theorem}
\newtheorem{lemma}[theorem]{Lemma}
\theoremstyle{definition}
\newtheorem{definition}[theorem]{Definition}
\newtheorem{remark}[theorem]{Remark}
\newtheorem{remarks}[theorem]{Remarks}
\newtheorem{example}[theorem]{Example}
\newtheorem{algorithm}[theorem]{Algorithm}
\numberwithin{theorem}{section}
\numberwithin{equation}{section}
\renewcommand{\o}{\mathfrak{o}}
\newcommand{\p}{\mathfrak{p}}
\providecommand{\C}{\mathbb{C}}
\renewcommand{\C}{\mathbb{C}}
\newcommand{\Z}{\mathbb{Z}}
\newcommand{\Q}{\mathbb{Q}}
\newcommand{\R}{\mathbb{R}}
\newcommand{\pmin}{p_{\text{min}}}
\newcommand{\pmax}{p_{\text{max}}}
\newcommand{\dmax}{d_{\text{max}}}
\newcommand{\zmax}{z_{\text{max}}}
\newcommand{\rhod}{\rho_{\text{div}}}
\newcommand{\rhos}{\rho_{\text{sol}}}
\newcommand{\rhoz}{\rho_{\text{ap}}}
\newcommand{\pid}{\pi_{\text{div}}}

\newcommand{\piz}{\pi_{\text{ap}}}
\newcommand{\CR}{\mathcal{C}}
\newcommand{\AR}{\mathcal{A}}
\newcommand{\ZR}{\mathcal{Z}}
\newcommand{\SR}{\mathcal{S}}
\DeclareMathOperator{\Disc}{Disc}
\DeclareMathOperator{\ord}{ord}
\DeclareMathOperator{\vol}{vol}
\DeclareMathOperator{\sgn}{sgn}

\pgfplotsset{compat=1.16}

\begin{document}
\title{On a question of Mordell}
\author{Andrew R. Booker}
\address{School of Mathematics, University of Bristol, Woodland Road, Bristol, BS8 1UG, UK}
\email{andrew.booker@bristol.ac.uk}

\author{Andrew V. Sutherland}
\address{Department of Mathematics, Massachusetts Institute of Technology, 77 Mass.\ Ave., Cambridge, MA \ 02139, USA}
\email{drew@math.mit.edu}

\begin{abstract}
We make several improvements to methods for finding
integer solutions to $x^3+y^3+z^3=k$ for small values of $k$.
We implemented these improvements on Charity Engine's global compute grid of
500,000 volunteer PCs and found new representations for
several values of $k$, including 3 and 42. This completes
the search begun by Miller and Woollett in 1954 and resolves a
challenge posed by Mordell in 1953.
\end{abstract}
\maketitle
\vspace{0pt}
\begin{center}
\parbox[c]{250pt}{\hspace{-3pt}\scriptsize{\textit{``I think the problem, to be quite honest with you, is that you've never actually known what the question is.'' -- Deep Thought}}}
\end{center}
\bigskip

\section{Introduction}
Let $k$ be an integer with $k\not\equiv\pm4\pmod*{9}$. Heath-Brown
\cite{Heath-Brown} has conjectured that there are infinitely many
triples $(x,y,z)\in\Z^3$ such that
\begin{equation}\label{eq:threecubes}
x^3+y^3+z^3=k.
\end{equation}
Interest in this Diophantine equation goes back at least to Mordell
\cite{Mordell}, who asked whether there are any solutions to
\eqref{eq:threecubes} for $k=3$ other than permutations of $(1,1,1)$ and
$(4,4,-5)$. The following year, Miller and Woollett \cite{MW} used the
\texttt{EDSAC} at Cambridge to run the first in a long line of computer
searches attempting to answer Mordell's question, and also expanded the
search to all positive $k\le 100$.

In this paper we build on the approach of the first author in
\cite{Booker}, and find the following new solutions
to \eqref{eq:threecubes}:
\begingroup\makeatletter\def\f@size{10}\check@mathfonts
\def\maketag@@@#1{\hbox{\m@th\large\normalfont#1}}%
\begin{align*}
569936821221962380720^3+(-569936821113563493509)^3+(-472715493453327032)^3&=3,\\
(-80538738812075974)^3+80435758145817515^3+12602123297335631^3&=42,\\
(-385495523231271884)^3+383344975542639445^3+98422560467622814^3&=165,\\
143075750505019222645^3+(-143070303858622169975)^3+(-6941531883806363291)^3&=579,\\
(-74924259395610397)^3+72054089679353378^3+35961979615356503^3&=906.
\end{align*}\endgroup
In particular, we answer Mordell's question and complete Miller and
Woollett's search by finding at least one solution to \eqref{eq:threecubes} for
all $k\le 100$ for which there are no local obstructions.

The algorithm used in \cite{Booker} is a refinement of an approach originally
suggested in \cite{HBLtR}, which is based on the following observation.
Let us first assume that $|x|>|y|>|z|$ and define
\[
d\coloneqq |x+y|.
\]
Then $d$ is nonzero, and the solutions
to \eqref{eq:threecubes} are precisely the triples $(x,y,z)$ for which $z$
is a cube root of $k$ modulo $d$ and the integer
\begin{equation}\label{eq:sdz}
\Delta(d,z)\coloneqq 3d(4|k-z^3|-d^3)
\quad\text{is a perfect square}.
\end{equation}
Solutions that do not satisfy $|x|>|y|>|z|$ can
be efficiently found by other means: after a suitable permutation
either $x=-y$ and $z$ is a cube root of~$k$, or $y=z$ and we seek a 
solution to the Thue equation $x^3+2y^3=k$, which can be easily handled.
If $|x|>|y|>|z|$ and we also assume $|z|>\sqrt{k}$ then we must have $0<d<\alpha|z|$,
where $\alpha=\sqrt[3]{2}-1\approx 0.25992$; see \cite[\S2]{Booker} for
details. Solutions with $|z|\le\sqrt{k}$
are easily found by solving $x^3+y^3=k-z^3$ for each fixed $z$ with $|z|\le\sqrt{k}$.

This leads to an algorithm that searches for solutions with $|z|\le B$ by
enumerating positive integers $d\le\alpha B$, and for each such $d$, determining
the residue classes of all cube roots of $k$ modulo $d$ and searching the corresponding
arithmetic progressions for values of $z\in[-B,B]$ that make
$\Delta(d,z)$ a square.
With suitable optimizations, including sieving arithmetic progressions to quickly rule out
integers that are not squares modulo primes in a suitably chosen set, this leads to
an algorithm that requires only $O\bigl(B(\log\log B)(\log\log\log B)\bigr)$
operations on integers
in $[0,B]$ for any fixed value of $k$.  An attractive feature of this algorithm
is that it finds all solutions with $\min\{|x|,|y|,|z|\}\le B$, even those for which
$\max\{|x|,|y|,|z|\}$ may be much larger than $B$ (note that this is the case in our
solution for $k=3$).

This algorithm was used in \cite{Booker} to find solutions for $k=33$ and $k=795$,
leaving only the following eleven $k\le 1000$ unresolved:
\begin{equation}\label{eq:openk}
42,\ 114,\ 165,\ 390,\ 579,\ 627,\ 633,\ 732,\ 906,\ 921,\ 975.
\end{equation}
The search in \cite{Booker} also ruled out any solutions for these $k$ with $\min\{|x|,|y|,|z|\}\le 10^{16}$.

Here we make several improvements to this method in \cite{Booker} that allow
us to find a new solution for $k=3$ as well as four of the outstanding $k$ listed above.
\begin{itemize}
\setlength\itemsep{4pt}
\item Instead of a single parameter $B$ bounding $|z|\le B$ and $0< d \le \alpha B$,
we use independent bounds $\dmax$ on $d$ and $\zmax$ on $|z|$, whose ratio we optimize
via an analysis of the expected distribution of $|z|/d$; this typically leads to a $\zmax/\dmax$ ratio
that is 10 to 20 times larger than the ratio $1/\alpha\approx 3.847332$ used in \cite{Booker}.
\item Rather than explicitly representing a potentially large set of sieved arithmetic progressions containing candidate values of $z$ for a given $d$, we implicitly represent them as intersections of arithmetic progressions modulo the prime power factors of $d$ and auxiliary primes.
This both improves the running time and reduces the memory footprint of the algorithm, allowing for much larger values of $|z|$.
\item We dynamically optimize the choice of auxiliary primes used for sieving based on the values of $k$ and $d$; when $d$ is much smaller than $\zmax$ this can reduce the number of candidate values of $z$ by several orders of magnitude.
\item We exploit $3$-adic and cubic reciprocity constraints for all
$k\equiv\pm3\pmod*{9}$; for the values of $k$ listed in~\eqref{eq:openk} this
reduces the average number of $z$ we need to check for a given value
of~$d$ by a factor of between 2 and 4 compared to the congruence
constraints used in \cite{Booker}, which did not use cubic reciprocity
for $k\ne 3$.
\end{itemize}
Along the way we compute to high precision the expected
density of solutions to \eqref{eq:threecubes} conjectured by
Heath-Brown \cite{Heath-Brown}, and compare it with the numerical
data compiled by Huisman \cite{Huisman} for $k\in[3,1000]$ and
$\max\{|x|,|y|,|z|\}\le10^{15}$. The data strongly support Heath-Brown's
conjecture that \eqref{eq:threecubes} has infinitely many solutions for
all $k\not\equiv\pm4\pmod*{9}$.

\subsection*{Acknowledgments}
We are extremely grateful to Charity Engine for providing the
computational resources used for this project, and in particular to
Mark McAndrew, Matthew Blumberg, and Rytis Slatkevi\v{c}ius, who were
responsible for running these computations on the Charity Engine compute
grid.

We thank Roger Heath-Brown for several stimulating discussions;
in particular, his conversation with Booker in the Nettle and Rye
on 27 February 2019 informed the analysis presented in
Section~\ref{sec:heuristics}.

Sutherland also acknowledges the support of the Simons Foundation
(award 550033).

\section{Density computations}
In this section we study Heath-Brown's conjecture in detail. In
particular, we explain how to compute the conjectured density of solutions
to high precision and compare the results with available numerical data. We further study the
densities of divisors $d\mid z^3-k$ and arithmetic progressions $z\pmod*{d}$
that occur in our algorithm, which informs the choice of
parameters used in our computations.

Let $k$ be a cubefree integer with $k\ge3$ and $k\not\equiv\pm4\pmod9$.
Define $K=\Q(\sqrt[3]{k})$ and $F=\Q(\sqrt{-3})$, and let $\o_K$ and
$\o_F$ be their respective rings of integers. We have
$\o_F=\Z[\zeta_6]$, where $\zeta_6=\frac{1+\sqrt{-3}}2$ is a
generator of the unit group $\o_F^\times$. Also, $\Disc(F)=-3$
and $\Disc(K)=-3f^2$, where, by \cite[Lemma~2.1]{Harron},
$$
f=\biggl(\prod_{p\mid k}p\biggr)\cdot
\begin{cases}
1&\text{if }k\equiv\pm1\pmod*{9},\\
3&\text{otherwise}.
\end{cases}
$$

We define two modular forms related to $F$ and $K$. First, let
$f_1$ be the modular form of weight $1$ and level $|\Disc(K)|$
such that $\zeta_K(s)=\zeta(s)L(s,f_1)$.
It follows from the ramification description in \cite[\S2.1]{Harron} that rational primes $p$ decompose into prime ideals
of $\o_K$ as follows (subscripts denote inertia degrees):
\[
p\o_K = \begin{cases}
\p_1\p_2 & \text{if $p\equiv 2\pmod*{3}$ and $p\nmid k$},\\
\p_1\p_1'\p_1'' & \text{if $p\equiv 1\pmod*{3}$ and $p\nmid k$ and $k$ is a cube modulo $p$},\\
\p_3 & \text{if $p\equiv 1\pmod*{3}$ and $p\nmid k$ and $k$ is not a cube modulo $p$},\\
\p_1^2\p_1' & \text{if $p=3$ and $k\equiv \pm 1\pmod*{9}$},\\
\p_1^3 & \text{otherwise}.
\end{cases}
\]
From this data we find that the local Euler factor of $L(s,f_1)$ at $p$ is
$$
L_p(s,f_1)=\frac1{1-c_p(k)p^{-s}+\left(\frac{\Disc(K)}{p}\right)p^{-2s}},
$$
where
$$
c_p(k)\coloneqq\begin{cases}
2&\text{if }p\nmid k,\,p\equiv1\pmod*{3}\text{ and }k^{(p-1)/3}\equiv1\pmod*{p},\\
-1&\text{if }p\nmid k,\,p\equiv1\pmod*{3}\text{ and }k^{(p-1)/3}\not\equiv1\pmod*{p},\\
1&\text{if }p=3\text{ and }k\equiv\pm1\pmod*{9},\\
0&\text{otherwise}.
\end{cases}
$$

Now let $\sigma:F\to\C$ be the unique embedding for which we have
$\Im\sigma(\sqrt{-3})>0$. Let
$\chi_f:(\o_F/3\o_F)^\times\to\C^\times$ be the character defined by
$\chi_f(\zeta_6+3\o_F)=\sigma(\zeta_6^{-1})$, and define
$\chi_\infty(z)\coloneqq\sigma(z)/|\sigma(z)|$.
Let $\chi$ be the Gr\"{o}ssencharakter of $F$ defined by
$$
\chi(\alpha\o_F)\coloneqq\begin{cases}
\chi_\infty(\alpha)\chi_f(\alpha+3\o_F)
&\text{if }\alpha\in\o_F\setminus\sqrt{-3}\o_F,\\
0&\text{if }\alpha\in\sqrt{-3}\o_F.
\end{cases}
$$
By automorphic induction, there is a holomorphic newform
$f_2$ of weight $2$ and level $|\Disc(F)|N(3\o_F)=27$ such that
$L(s,f_2)=L(s-\frac12,\chi)$.

Given a prime $p\equiv1\pmod*{3}$, let $a_p$ denote the unique integer
for which $a_p\equiv1\pmod*{3}$ and $4p=a_p^2+27b^2$ for some
$b\in\Z_{>0}$.
Let $\alpha=\frac{a_p+3b\sqrt{-3}}{2}$ and $\p=\alpha\o_F$, so that
$p\o_F=\p\overline{\p}$. We have
$$
\alpha+1=\frac{a_p+2+3b\sqrt{-3}}{2}
=3\left(\frac{\frac{a_p+2}{3}+b\sqrt{-3}}{2}\right)\in3\o_F,
$$
so $\chi_f(\alpha)=\chi_f(-1)=-1$. Thus,
$\chi(\p)=-\frac{\sigma(\alpha)}{\sqrt{p}}$ and
$\chi(\overline{\p})=-\frac{\overline{\sigma(\alpha)}}{\sqrt{p}}$,
so the Euler factor of $L(s,f_2)$ at $p$ (in its arithmetic normalization) is
$$
\frac1{(1-\chi(\p)p^{\frac12-s})(1-\chi(\overline{\p})p^{\frac12-s})}
=\frac1{1+a_pp^{-s}+p^{1-2s}}.
$$
For a prime $p\equiv2\pmod*3$, we have
$\chi(p\o_F)=\chi_\infty(p)\chi_f(p)=\chi_f(-1)=-1$, so the corresponding
Euler factor is
$$
\frac1{1-\chi(p\o_F)N(p\o_F)^{\frac12-s}}=\frac1{1+p^{1-2s}}.
$$
Finally, $\chi(\sqrt{-3})=0$, so the Euler factor at $p=3$ is $1$.

In summary, if we extend the definition of $a_p$ so that $a_p=0$
for $p\not\equiv1\pmod*{3}$, then the Euler factor of $L(s,f_2)$ at $p$
is
$$
L_p(s,f_2)=\frac1{1+a_pp^{-s}+\left(\frac{9}{p}\right)p^{1-2s}}.
$$

\subsection{Solution density}\label{sec:solutiondensity}
Define
$$
\sigma_p\coloneqq\lim_{e\to\infty}
\frac{\#\{(x,y,z)\pmod*{p^e}:x^3+y^3+z^3\equiv k\pmod*{p^e}\}}{p^{2e}}.
$$
Then, as calculated by Heath-Brown \cite{Heath-Brown}, we have
$$
\sigma_p=\begin{cases}
1+\frac{3c_p(k)}{p}-\frac{a_p}{p^2}
&\text{if }p\nmid 3k,\\
1+\frac{(p-1)a_p-1}{p^2}
&\text{if }p\mid k\text{ and }p\ne3,\\
\frac13\#\{(x,y,z)\pmod*{3}:x^3+y^3+z^3\equiv k\pmod*{9}\}&\text{if }p=3.
\end{cases}
$$

Now let $h\colon \R^3\to\R_{\ge0}$ be a height function, by which we
mean a function that is continuous,
symmetric in its inputs, and satisfies
\begin{itemize}
\item $h(x,y,z)>0$ when $x^3+y^3+z^3=0$ and $xyz\ne0$,
\item $h(\lambda x,\lambda y,\lambda z)=|\lambda|h(x,y,z)$
for any $\lambda\in\R^\times$.
\end{itemize}
The real density of solutions to $x^3+y^3+z^3=0$ with height in the interval $\mathcal H\coloneqq [H_1,H_2]$
can then be computed as follows.

For $\varepsilon > 0$ define
\begin{align*}
S(\varepsilon) &\coloneqq \left\{(x,y,z)\in\R^3: h(x,y,z)\in\mathcal H,\,|x^3+y^3+z^3|\le\varepsilon\right\},\\
T(\varepsilon) &\coloneqq \left\{(x,y,z)\in\R^3:x\ge y\ge z\ge0,\,h(x,-y,-z)\in\mathcal H,\,|x^3-y^3-z^3|\le\varepsilon\right\},
\end{align*}
so that $\vol(S(\varepsilon))/\vol(T(\varepsilon))\to12$ as
$\varepsilon\to0^+$.  We may then compute

\begin{equation}\label{eq:realdensity}
\begin{aligned}
\lim_{\varepsilon\to0^+}
(2\varepsilon)^{-1}\vol(S(\varepsilon)) &= 12\lim_{\varepsilon\to0^+}(2\varepsilon)^{-1}\vol(T(\varepsilon))\\
&=12\int_0^\infty\int_z^\infty
\textbf{1}_{h(\sqrt[3]{y^3+z^3},-y,-z)\in\mathcal H}
\frac{dy}{3(y^3+z^3)^{2/3}}\,dz\\
&=4\int_0^\infty\int_1^\infty
\textbf{1}_{zh(\sqrt[3]{t^3+1},-t,-1)\in\mathcal H}
\frac{dt}{(t^3+1)^{2/3}}\frac{dz}{z}\\
&=4\int_1^\infty\int_{H_1/h(\sqrt[3]{t^3+1},-t,-1)}^{H_2/h(\sqrt[3]{t^3+1},-t,-1)}
\frac{dz}{z}\frac{dt}{(t^3+1)^{2/3}}
=\sigma_\infty\log\frac{H_2}{H_1},
\end{aligned}
\end{equation}
where $\sigma_\infty=4\int_1^\infty(t^3+1)^{-2/3}\,dt=
\frac23\frac{\Gamma(\frac13)^2}{\Gamma(\frac23)}$.

Heath-Brown conjectures that the number $n(B)$ of solutions to
\eqref{eq:threecubes}, up to permutation,
satisfying $\max\{|x|,|y|,|z|\}\le B$ is asymptotic to
$$
e(B)\coloneqq \rhos\log{B}\quad\text{as }B\to\infty,
\qquad\text{where }\rhos\coloneqq \tfrac16\sigma_\infty\prod_p\sigma_p.
$$
As shown above, the real density does not depend on the precise choice
of the height function~$h$.  We thus conjecture that the same asymptotic density applies to the
solutions satisfying $h(x,y,z)\le B$ for any similar choice of $h$,
including, for example,
$$
\min\{|x|,|y|,|z|\},\quad
|xyz|^{\frac13},\quad
\text{and}\quad
d=\min\{|x+y|,|x+z|,|y+z|\}.
$$

Let us now define
$$
r_p\coloneqq \frac{\sigma_p}
{L_p(1,f_1)^3L_p(2,f_2)L_p(2,f_1)^{-6}\zeta_p(2)^{-6}L_p(2,(\tfrac{\cdot}{3}))^{-3}},
$$
where
$$
\zeta_p(s)=\frac1{1-p^{-s}}
\quad\text{and}\quad
L_p(s,(\tfrac{\cdot}{3}))=\frac1{1-\left(\frac{p}{3}\right)p^{-s}}.
$$
A straightforward calculation shows that
$$
r_p=1-\frac{3a_pc_p(k)+O(1)}{p^3}.
$$
Since $-a_pc_p(k)$ is the coefficient of $p^{-s}$ in the
Rankin--Selberg $L$-function $L(s,f_1\boxtimes f_2)$, we expect
square-root cancellation in the product $\prod_pr_p$. Under the generalized Riemann hypothesis (GRH),
for large $X$ we have
\begin{equation}\label{eq:rpest}
\prod_p\sigma_p=\bigl(1+O(X^{-2}\log{X})\bigr)
L(1,f_1)^3L(2,f_2)L(2,f_1)^{-6}\zeta(2)^{-6}L(2,(\tfrac{\cdot}{3}))^{-3}
\prod_{p\le X}r_p.
\end{equation}

%\subsubsection{Testing Heath-Brown's conjecture}
Applying \eqref{eq:rpest} with $X=10^9$ allows us to compute the solution
densities $\rhos$ to roughly 18 digits of precision for all cubefree
$k\le1000$. To evaluate the $L$-functions, we used the extensive
functionality available for that purpose in \texttt{PARI/GP}~\cite{PARI2}.
Since our goal is merely to gather some statistics, we
content ourselves with a heuristic estimate of the error in this
computation, although it could be rigorously certified with more work.
Some examples are shown in Table~\ref{tab:rhos}.

\begin{table}[tbh!]
\small
\begin{tabular}{rrrrrrrrrrrr}
&&&&\multicolumn{2}{c}{$B=10^5$}&&\multicolumn{2}{c}{$B=10^{10}$}&&\multicolumn{2}{c}{$B=10^{15}$}\\
\cmidrule{5-6}\cmidrule{8-9}\cmidrule{11-12}
$k$ & $\rhos$ & $\lceil \exp (1/\rhos)\rceil$ && $e(B)$ & $n(B)$ && $e(B)$ & $n(B)$ && $e(B)$ & $n(B)$\\\midrule
$858$ & $0.028504$ & $1723846985902459$ && $0.328$ & $1$ && $0.656$ & $2$ && $0.984$ & $2$\\
$276$ & $0.031854$ & $43031002119138$ && $0.367$ & $1$ && $0.733$ & $1$ && $1.100$ & $2$\\
$390$ & $0.032935$ & $15358736844736$ && $0.379$ & $0$ && $0.758$ & $0$ && $1.138$ & $0$\\
$516$ & $0.033062$ & $13665771588173$ && $0.381$ & $0$ && $0.761$ & $1$ && $1.142$ & $1$\\
$663$ & $0.033196$ & $12097471969974$ && $0.382$ & $0$ && $0.764$ & $1$ && $1.147$ & $1$\\
$975$ & $0.038722$ & $164297126902$ && $0.446$ & $0$ && $0.892$ & $0$ && $1.337$ & $0$\\
$165$ & $0.039636$ & $90602378809$ && $0.456$ & $0$ && $0.913$ & $0$ && $1.369$ & $0$\\
$555$ & $0.042706$ & $14770444441$ && $0.492$ & $1$ && $0.983$ & $2$ && $1.475$ & $2$\\
$921$ & $0.044142$ & $6895540744$ && $0.508$ & $0$ && $1.016$ & $0$ && $1.525$ & $0$\\
$348$ & $0.044632$ & $5378175303$ && $0.514$ & $2$ && $1.028$ & $2$ && $1.542$ & $3$\\
$906$ & $0.049745$ & $537442063$ && $0.573$ & $0$ && $1.145$ & $0$ && $1.718$ & $0$\\
$579$ & $0.050838$ & $348939959$ && $0.585$ & $0$ && $1.171$ & $0$ && $1.756$ & $0$\\
$114$ & $0.058459$ & $26853609$ && $0.673$ & $0$ && $1.346$ & $0$ && $2.019$ & $0$\\
$3$ & $0.061052$ & $12985612$ && $0.703$ & $2$ && $1.406$ & $2$ && $2.109$ & $2$\\
$732$ & $0.063137$ & $7561540$ && $0.727$ & $0$ && $1.454$ & $0$ && $2.181$ & $0$\\
$633$ & $0.079660$ & $283059$ && $0.917$ & $0$ && $1.834$ & $0$ && $2.751$ & $0$\\
$33$ & $0.088833$ & $77422$ && $1.023$ & $0$ && $2.045$ & $0$ && $3.068$ & $0$\\
$795$ & $0.089491$ & $71273$ && $1.030$ & $0$ && $2.061$ & $0$ && $3.091$ & $0$\\
$42$ & $0.113449$ & $6732$ && $1.306$ & $0$ && $2.612$ & $0$ && $3.918$ & $0$\\
$627$ & $0.129565$ & $2249$ && $1.492$ & $0$ && $2.983$ & $0$ && $4.475$ & $0$\\

\bottomrule
\end{tabular}
\bigskip

\caption{Selected $\rhos$ and $\lceil \exp (1/\rhos)\rceil = \min\{B\in
\Z:e(B)\ge 1\}$ values for $k\le 1000$, including the ten smallest $\rhos$
and all $k$ with $n(10^{15})=0$.}\label{tab:rhos}
\end{table}

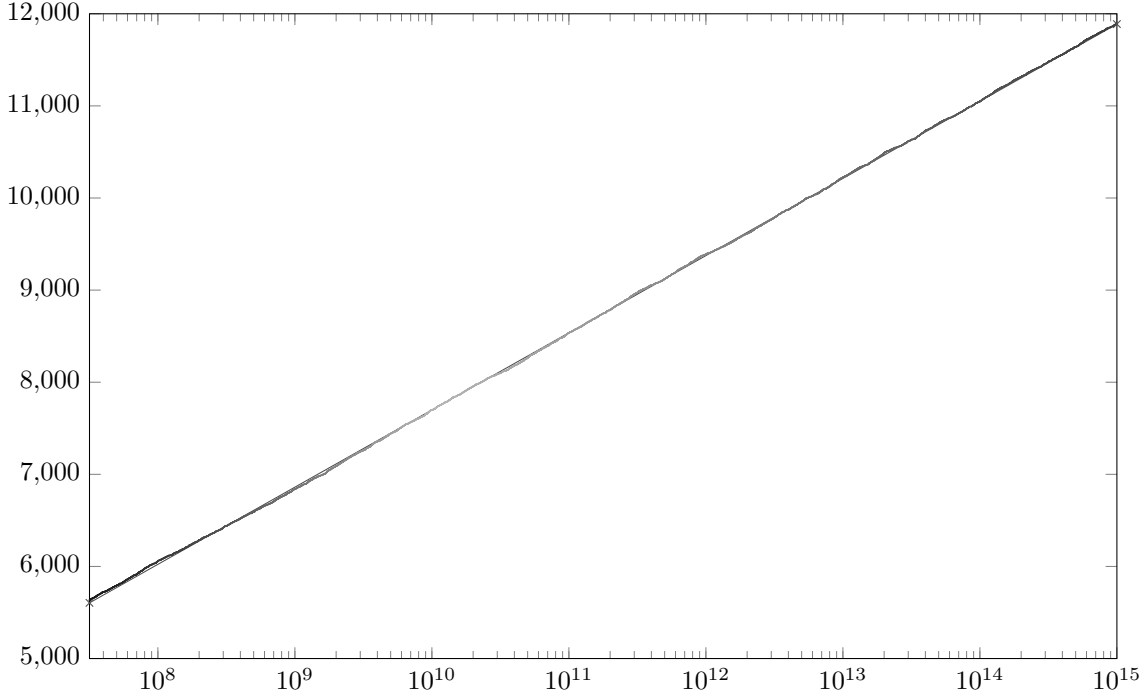
\begin{figure}[tbh!]
\begin{tikzpicture}
\selectcolormodel{gray}
\begin{semilogxaxis}[
width=6in,
height=4in,
enlargelimits=false,
ytick={5000,6000,7000,8000,9000,10000,11000,12000},
ymin=5000,
ymax=12000,
yticklabel style={
    /pgf/number format/fixed,
    /pgf/number format/precision=5
},
scaled y ticks=false
]
\addplot+[only marks, scatter, mark=x, mark size=0.01cm] table{huismanplot.dat};
\addplot+[mark=x] coordinates {
	(31681904,5604.391)
	(1000000000000000,11888.181)
};
\end{semilogxaxis}
\end{tikzpicture}
\caption{Scatter plot of $N_{1000}(B)$ as a function of
$B\in [10^{7.5},10^{15}]$ based on Huisman's dataset \cite{Huisman},
compared to the line $\rho_{1000}\log B +C$.}\label{fig:density}
\end{figure}

We compared Huisman's data set to an average form of
Heath-Brown's conjecture as follows.
For an integer $K\ge3$, define
$$
N_K(B)\coloneqq \#\{(k,x,y,z)\in\Z^4:x^3+y^3+z^3=k\text{ cubefree},\,3\le k\le K,\,|z|\le|y|\le|x|\le B\}
$$
and
$$
\rho_K\coloneqq \sum_{\substack{k\in\Z\cap[3,K]\\k\text{ cubefree}}}\rhos(k).
$$
Then Heath-Brown's conjecture implies that for fixed $K$, we have
$N_K(B)\sim\rho_K\log{B}$ as $B\to\infty$.
The plot in Figure~\ref{fig:density} compares $N_{1000}(B)$ for $B\in[10^{7.5},10^{15}]$,
computed from Huisman's data \cite{Huisman}, with $\rho_{1000}\log{B}+C$,
where $\rho_{1000}\approx363.869$ and $C\approx-679.4$ was chosen to
minimize the mean square difference.  Out of 6256 points, the two plots
never differ by more than 42, which gives strong evidence for Heath-Brown's
conjecture, at least on average over $k$.

\subsection{Divisor and arithmetic progression densities}
In this section we assume that $k\equiv\pm3\pmod*9$ and derive estimates for
the density of arithmetic progressions arising from cube roots of $k$ modulo
$d$. Define
$$
\delta_d\coloneqq \begin{cases}
1&\text{if }\exists z\in\Z\text{ s.t.\ }z^3\equiv k\pmod*{d}\text{ and }
\ord_p(d)\in\{0,\ord_p(k/3)\}\;\forall p\mid k,\\
0&\text{otherwise},
\end{cases}
$$
and
$$
F(s)\coloneqq \sum_{d=1}^\infty\frac{\delta_d}{d^s}.
$$
As shown in \cite{Booker}, any $d$ arising from a solution to
\eqref{eq:threecubes} must satisfy $\delta_d=1$, and we only consider
such $d$ in our algorithm.
\smallskip

For $p\nmid k$ and $e>0$, we have
$\delta_{p^e}=\frac{c_p(k)+2-\left(\frac{p}{3}\right)}{3}$, so that
$F(s)=\prod_pF_p(s)$, where
$$
F_p(s)\coloneqq\begin{cases}
\Bigl(1-\frac{c_p(k)+2-\left(\frac{p}{3}\right)}{3p^s}\Bigr)^{-1}
&\text{if }p\nmid k,\\
1+p^{-\ord_p(k)s}&\text{if }p\mid\frac{k}{3},\\
1&\text{if }p=3.
\end{cases}
$$
For $p\nmid k$, the local factor
$\frac{F_p(s)^3L_p(s,(\frac{\cdot}{3}))}{\zeta_p(s)^2L_p(s,f_1)}$
is $1+O(p^{-3s})$. Therefore, $F(s)^3$ has meromorphic continuation to
$\Re(s)>\frac13$, with a pole of order $2$ at $s=1$ and no other poles
in the region $\{s\in\C:\Re(s)\ge1\}$.
By \cite[Theorem~3.1]{Kato}, it follows that
$$
\sum_{d\le\dmax}\delta_d\sim
\rhod\frac{\dmax}{\sqrt[3]{\log\dmax}}
\quad\text{as }\dmax\to\infty,
\qquad\text{where }
\rhod\coloneqq\frac{\bigl(\lim_{s\to1}F(s)^3(s-1)^2\bigr)^{\frac13}}{\Gamma(\frac23)}.
$$
In turn, we have
$$
\lim_{s\to1}F(s)^3(s-1)^2=\bigl(1+O(X^{-2})\bigr)
\frac{L(1,f_1)}{L(1,(\frac{\cdot}{3}))}
\prod_{p\le X}\frac{F_p(1)^3L_p(1,(\frac{\cdot}{3}))}{\zeta_p(1)^2L_p(1,f_1)}.
$$

Let us now define
$$
G(s)\coloneqq \sum_{d=1}^\infty\frac{\delta_dr_d(k)}{d^s},
\quad\text{where }
r_d(k)=\#\{z\pmod*{d}:z^3\equiv k\pmod*{d}\}.
$$
Then $G(s)=\prod_pG_p(s)$, where
$$
G_p(s)\coloneqq \begin{cases}
1+\frac{1+c_p(k)}{p^s-1}&\text{if }p\nmid k,\\
1+p^{(1-s)\ord_p(k)-1}&\text{if }p\mid\frac{k}{3},\\
1&\text{if }p=3.
\end{cases}
$$
For $p\nmid k$ we have
$\frac{G_p(s)L_p(2s,f_1)\zeta_p(2s)}{L_p(s,f_1)\zeta_p(s)}=1+O(p^{-3s})$.
Therefore,
$$
\sum_{d\le\dmax}\delta_dr_d(k)\sim\rhoz\dmax
\quad\text{as }\dmax\to\infty,
\qquad\text{where }
\rhoz\coloneqq\lim_{s\to1}G(s)(s-1).
$$
In turn, we have
$$
\rhoz=\bigl(1+O(X^{-2})\bigr)\frac{L(1,f_1)}{L(2,f_1)\zeta(2)}
\prod_{p\le X}\frac{G_p(1)L_p(2,f_1)\zeta_p(2)}{L_p(1,f_1)\zeta_p(1)}.
$$

Table~\ref{tab:rhoap} lists estimates $\rhoz\dmax$ for the number
$\piz(\dmax)$ of arithmetic progressions modulo $d\le\dmax$,
as well as estimates $\rhod\dmax/\sqrt[3]{\log\dmax}$ for the number
$\pid(\dmax)$ of admissible $d\le\dmax$, along with the ratios of
these quantities.

\begin{table}[tbh!]
\small
\begin{tabular}{rrrrrrr}
$k$ & $\rhoz\dmax$ & $\piz(\dmax)$ & $\frac{\rhod\dmax}{\sqrt[3]{\log\dmax}}$ & $\pid(\dmax)$ & $\frac{\rhoz\sqrt[3]{\log \dmax}}{\rhod}$ &$\frac{\piz(\dmax)}{\pid(\dmax)}$\\\midrule
  $3$ & $476709085641$ & $476709082386$ & $221480415360$ & $222316170600$ & $2.152$ & $2.144$\\
 $42$ & $439262042312$ & $439262055314$ & $194525166395$ & $195043114314$ & $2.258$ & $2.252$\\
$114$ & $346031225026$ & $346031232985$ & $169944552313$ & $169697769695$ & $2.036$ & $2.039$\\
$165$ & $398768628911$ & $398768635237$ & $201820401130$ & $201648107384$ & $1.976$ & $1.978$\\
$390$ & $361424697190$ & $361424750258$ & $170411108873$ & $170119932464$ & $2.121$ & $2.125$\\
$579$ & $467532879762$ & $467532936236$ & $220746986113$ & $221627128720$ & $2.118$ & $2.110$\\
$627$ & $544308148137$ & $544308117802$ & $238234806279$ & $240026258762$ & $2.285$ & $2.268$\\
$633$ & $510771397972$ & $510771391669$ & $227368579096$ & $228697959163$ & $2.246$ & $2.233$\\
$732$ & $396862883895$ & $396862943789$ & $145013347786$ & $145167910326$ & $2.737$ & $2.734$\\
$906$ & $353110285004$ & $353110236539$ & $166128603588$ & $165813813631$ & $2.126$ & $2.130$\\
$921$ & $420143131383$ & $420143101621$ & $212693499876$ & $212924474063$ & $1.975$ & $1.973$\\
$975$ & $461977372770$ & $461977396756$ & $194140103965$ & $194481735572$ & $2.380$ & $2.375$\\\bottomrule
\end{tabular}
\bigskip

\caption{Comparison of estimated and actual counts of arithmetic progressions modulo $d\le \dmax=10^{12}$ for various $k$ of interest.}\label{tab:rhoap}
\end{table}

\begin{remark}
The average number of arithmetic progressions modulo $d\le \dmax$ listed in Table~\ref{tab:rhoap} is strikingly small.  Even for $\dmax=10^{24}$, which is well beyond the feasible range, the average is around 3 and never above 3.5 for any of the listed $k$.
\end{remark}

\begin{remark}\label{rem:rhoap}
For any fixed choice of the ratio $R=\zmax/\dmax$, the total running time of our algorithm is roughly proportional to $\rhoz\dmax$.
The constant of proportionality can be estimated by running the algorithm on a suitable sample of
$d\le \dmax$. These estimates allow us to efficiently manage resource allocation
in large distributed computations; see Section~\ref{sec:computation} for details.
\end{remark}

\section{Cubic reciprocity}
In \cite{Cassels}, Cassels used cubic reciprocity to prove
that whenever $x,y,z\in\Z$ satisfy
$x^3+y^3+z^3=3$ we must have $x\equiv y\equiv z\pmod*{9}$.
For fixed $d=|x+y|$ it follows that $z$ is determined modulo~$81$.
Colliot-Th\'el\`ene and Wittenberg \cite{CW} later recast this phenomenon in
terms of Brauer--Manin obstructions, and showed that for any
$k$, the solutions to \eqref{eq:threecubes} are always forbidden for
some residue classes globally but not locally.\footnote{Thus strong
approximation fails for \eqref{eq:threecubes}, but
this is never enough to forbid the existence of integer solutions
outright, so there is no Brauer--Manin obstruction.}
In this section we extend Cassels' analysis to all cubefree
$k\equiv\pm3\pmod*{9}$, and derive constraints on the residue class of
$z\pmod*{q}$ for a certain modulus $q\mid 27k$.
We assume throughout that $k\equiv3\epsilon\pmod*9$ for a fixed
$\epsilon\in\{\pm1\}$.

Given $\alpha,\beta\in\o_F$ with $\beta\notin\sqrt{-3}\o_F$, let
$\left(\frac{\alpha}{\beta}\right)_3$ be the cubic residue symbol, as defined
in \cite[Ch.\ 9 and 14]{IR}.
Put $\zeta_3=\frac{-1+\sqrt{-3}}2\in\o_F$.
For integers $x,y$ satisfying $x\equiv y\equiv\epsilon\pmod*{3}$, define
$$
\chi_k(x,y)\coloneqq \zeta_3^{\epsilon(y-x)/3}
\left(\frac{\zeta_3x+\zeta_3^{-1}y}{k/3}\right)_3.
$$
Note that $\chi_k(x,y)$ depends only on the residue classes of $x,y\pmod*{3k}$.

\begin{definition}\label{def:admissible}
We say that a pair $(d,z)\in\Z^2$ is \emph{admissible} if
there exist $x,y\in\Z$ satisfying the following conditions:
\begin{enumerate}
\item $x+y\equiv-\epsilon\left(\frac{d}{3}\right)d\pmod*{27k}$;
\item $x^3+y^3+z^3\equiv k\pmod*{81k}$;
\item $\{\chi_k(x,y),\chi_k(x,z),\chi_k(y,z)\}\subseteq\{0,1\}$.
\end{enumerate}
\end{definition}
Note that this definition depends only on the residue classes of
$d,z\pmod*{27k}$.
\begin{lemma}
Let $(x,y,z)\in\Z^3$ be a solution to \eqref{eq:threecubes},
and let $d\coloneqq |x+y|$. Then $(d,z)$ is admissible.
\end{lemma}
\begin{proof}
Since every cube is congruent to $0$ or $\pm1\pmod*9$, we have $x\equiv
y\equiv z\equiv\epsilon\pmod*3$, so that
$x+y\equiv-\epsilon\equiv-\epsilon\left(\frac{d}{3}\right)d\pmod*3$.
As $d=|x+y|$, it follows that $x+y=-\epsilon\left(\frac{d}{3}\right)d$,
so condition (1) of the definition is satisfied. Condition (2) then
follows directly from \eqref{eq:threecubes}.

Now let
$$
\gamma\coloneqq \epsilon(\zeta_3x+\zeta_3^{-1}y)=-\epsilon y+\epsilon(x-y)\zeta_3.
$$
By \cite[Ch.~9, Ex.~19]{IR}, we have
\begin{align*}
\chi_k(x,y)=\zeta_3^{\epsilon(y-x)/3}
\left(\frac{\zeta_3x+\zeta_3^{-1}y}{k/3}\right)_3
=\left(\frac{3}{\gamma}\right)_3\left(\frac{\epsilon\gamma}{k/3}\right)_3
=\left(\frac{-3\epsilon}{\gamma}\right)_3\left(\frac{\gamma}{-\epsilon k/3}\right)_3,
\end{align*}
where the last equality follows from the fact that $(\alpha/\beta)_3$
depends only on the ideal $\beta\o_F$ and $(\pm\epsilon/\beta)_3=((\pm\epsilon)^3/\beta)_3=1$.
By cubic reciprocity \cite[Ch.~14, Theorem~1]{IR}, this equals
$$
\left(\frac{-3\epsilon}{\gamma}\right)_3
\left(\frac{-\epsilon k/3}{\gamma}\right)_3
=\left(\frac{k}{\gamma}\right)_3.
$$
Noting that $x^3+y^3=(x+y)\gamma\overline{\gamma}$, we have
$k\equiv z^3\pmod*{\gamma\o_F}$, whence
$$
\chi_k(x,y)=\left(\frac{z^3}{\gamma}\right)_3\in\{0,1\},
$$
and by symmetry, we also have $\chi_k(x,z),\chi_k(y,z)\in\{0,1\}$; thus condition (3) holds as well.
\end{proof}

\begin{lemma}\label{lem:admissibility}
Let
$$
q\coloneqq 27k\!\!\!\!\!\!\!\!\!\!\!\prod_{\substack{p\mid k\\\ord_p(k)=2\\p=2\text{ or }c_p(2)=-1}}\!\!\!\!\!\!\!\!\!\!\!\!\!p^{-1}\ ,
$$
and let $d,z,z'\in\Z$ satisfy $z'\equiv z\pmod*{q}$.  Then
$(d,z)$ is admissible iff $(d,z')$ is admissible.
\end{lemma}
\begin{proof}
Suppose that $(d,z)$ is admissible.
Let $p$ be a prime divisor of $27k/q$, and
consider $z'\equiv z\pmod*{27k/p}$.
By the Chinese remainder theorem
it suffices to show that $(d,z')$ is admissible in this case.

Set $a=(z'-z)/p$, so that $z'=z+ap$.
Let $x,y$ be integers satisfying the conditions in
Definition~\ref{def:admissible}, and
let $x'=x+bp$, $y'=y-bp$ for some $b\in27kp^{-2}\Z$.
Then
$$
(x')^3+(y')^3+(z')^3\equiv x^3+y^3+z^3+3p[az^2+b(x^2-y^2)]
\equiv 3p[az^2+b(x^2-y^2)]\pmod*{p^2}.
$$

If $p\mid(x^2-y^2)$ and $p\nmid(x+y)$ then we have
$x\equiv y\pmod*{p}$, $p>2$ and $p\nmid x$, which means that
$2x^3+z^3\equiv0\pmod*{p}$ and $2\equiv(-z/x)^3\pmod*{p}$ is a cubic residue mod $p$.
But $p>2$ implies $c_p(2)=-1$, meaning $p\equiv 1\pmod* 3$ and $2^{(p-1)/3}\not\equiv 1\pmod* p$, so $2$ cannot be a cubic residue mod $p$
and we must have $p\nmid(x^2-y^2)$ or $p\mid(x+y)$.

If $p\nmid(x^2-y^2)$ then we may choose $b$ so that
$b(x^2-y^2)\equiv -az^2\pmod*{p}$, while if $p\mid(x+y)$ then
$p\mid z$ and any choice of $b$ suffices.
It follows that
$$
(x')^3+(y')^3+(z')^3\equiv k\pmod*{81k}.
$$
Moreover, we have
$\chi_k(x',y')=\chi_k(x,y)$, $\chi_k(x',z')=\chi_k(x,z)$ and
$\chi_k(y',z')=\chi_k(y,z)$, by inspection. Thus $(d,z')$ is admissible, as desired.
\end{proof}
Thus, the definition of admissibility factors through
$\Z/27k\Z\times\Z/q\Z$.

\begin{example}
The following table shows the ratio
$$
\frac{\sum_{d\pmod*{27k}}\#\{z\pmod*{q}:(d,z)\text{ is admissible}\}}
{\sum_{d\pmod*{27k}}\#\{z\pmod*{q}:\exists x\pmod*{q}\text{ s.t.\ }
x^3+(d-x)^3+z^3\equiv k\pmod*{3q}\}},
$$
which is the average density of admissible residues
$z\pmod*{q}$ among all locally permitted residues,
for a few $k$ of interest:
\begin{center}
\begin{tabular}{r|rrrrr}
$k$ & $3$ & $33$ & $42$ & $114$ & $633$\\ \hline
density & $0.250$ & $0.590$ & $0.970$ & $0.962$ & $0.585$
\end{tabular}
\end{center}
Although the improvement is modest for some
$k$, those cases still benefit from imposing local constraints mod $q$, some 
of which were not used in \cite{Booker}; in particular, passing from mod
$9$ solutions to mod $81$ solutions reduces the density by a factor of $4/9$.
\end{example}

\subsection{Algorithm}\label{sec:algorithm}
Let $k\equiv 3\epsilon \pmod*{9}$ be a positive integer, and for each
positive integer~$m$, let
\[
\CR(m)\coloneqq \{z+m\Z:z^3\equiv k\pmod*{m}\} \subseteq\Z/m\Z
\]
denote the set of cube roots of $k$ modulo $m$.  Let $P$ be the set of
primes $p\nmid k$ for which $\#\CR(p)>0$; for $p\in P$ we then have
$\#\CR(p)=3$ if $p\equiv 1\pmod*3$ and $\#\CR(p)=1$ otherwise.

Let $A$ be a set of small auxiliary primes $p\nmid k$ whose product
exceeds $\dmax\zmax$; in practical computations we may take $A$ to be
the primes $p<256$ not dividing $k$.
Let $s\coloneqq\epsilon\left(\frac{d}{3}\right)$, so that any solution to
\eqref{eq:threecubes} with $d=|x+y|$ has $\sgn z=s$, 
and for positive integers $d$ and primes $p\nmid dk$ define
\[
\SR_d(p)\coloneqq
\begin{cases}
\{z+p\Z:3d(4s(z^3-k)-d^3)\equiv\square\pmod*{p}\}&\text{if }p>2,\\
\{k+d+2\Z\}&\text{if }p=2.
\end{cases}
\]

Finally, let $c_1>c_0>1$ and $c_2>1$ denote integers that we will choose
to optimize performance (typically $c_0\approx 4$, $c_1\approx 50$, and
$c_2\approx 6$), and let $q$ be the divisor of $27k$ defined in Lemma~\ref{lem:admissibility}.
\bigskip

\begin{algorithm}\label{alg}
Given $k,d_{\max},z_{\max}\in \Z_{>0}$ with $k\equiv3\epsilon\pmod*9$,
enumerate all pairs $(d,z)\in\Z^2$ for which there exist $(x,y,z)\in\Z^3$ satisfying \eqref{eq:threecubes}
with $|x|>|y|>|z|$, $\sqrt{k}<|z|\le\zmax$, and $|x+y|=d\le\dmax$ as follows:
\medskip

\noindent
Recursively enumerate all positive integers
$d_0=p_1^{e_1}\cdots p_n^{e_n}\le\dmax$, where $p_1>\cdots>p_n$ are
primes in $P$ and $e_i\in \Z_{>0}$. For each such $d_0$ do the following:
\vspace{3pt}
\begin{enumerate}[1.]
\setlength{\itemsep}{3pt}
\item For each positive divisor $d_1$ of $k/3$ with $\gcd(d_1,k/d_1)=1$, set
$d\coloneqq d_0d_1$ and let $\AR_d(q)$ be the set of $z+q\Z$ for which
$(d,z)$ is admissible.
\item Set $a\coloneqq 1$, and if $c_1 qd_0<\zmax$ then order the
$p\nmid d$ in $A$ by $\log\#\SR_d(p)/\log p$, and while $c_0qd_0pa<\zmax$
replace $a$ by $pa$, where $p$ is the next prime in the ordering.

\item Let $b$ be the product of $c_2$ primes $p\in A$ not dividing $da$,
chosen either using the ordering computed in the previous step or a fixed order.

\item Set $m\coloneqq d_0qa$, and let $\ZR(m)$ be the subset of
$\Z/m\Z$ that is identified with
\[
\CR(p_1^{e_1})\times\cdots\times\CR(p_n^{e_n})
\times\AR_d(q)\times\prod_{p\mid a}\SR_d(p)
\]
via the Chinese remainder theorem.
Let
\[\ZR(m,s,\zmax)\coloneqq\{z\in\Z:z+m\Z\in\ZR(m),\ \sgn{z}=s,
\text{ and } |z|\le \zmax\}.
\]
For each $z\in\ZR(m,s,\zmax)$, if $z+p\Z$ lies in $\SR_d(p)$ for all
$p\mid b$, check if $\Delta(d,z)$ is square, and if so output the pair $(d,z)$.
\end{enumerate}
\end{algorithm}

\begin{remarks}\label{rem:parallel}
The following remarks apply to the implementation of Algorithm~\ref{alg}.
\begin{itemize}
\item
The algorithm can be easily parallelized by restricting the range of $p_1$
and, for very small values of $p_1$, fixing $p_1$ and restricting the range
of $p_2$.

\item
The recursive enumeration of $d_0=p_1^{e_n}\cdots p_n^{e_n}$ ensures that typically only the
value of $p_n^{e_n}$ changes from one $d_0$ to the next, allowing the product
$\CR(p_1^{e_1})\times\cdots\times\CR(p_n^{e_n})$ to be updated incrementally
rather than recomputed for each $d_0$.

\item
The sets $\CR(p^e)$ are precomputed for $p\le\sqrt{\dmax}$,
as are the sets $\AR_d(q)$ for each $d\in \{1,\ldots q-1\}$ not divisible by 3,
and the sets $\SR_d(p)$ for each $p\in A$ and $d\in \{1,\ldots,p-1\}$.
This allows the sets $\ZR(m)$ to be efficiently enumerated using an explicit form of the Chinese remainder theorem that requires very little space.
We shall refer to this procedure as CRT enumeration.

\item
For $p\in A$ the precomputed sets $\SR_d(p)$ for $d\in\{1,\ldots,p-1\}$ are also
stored as bitmaps, as are Cartesian products of pairs
of these sets and certain triples; this facilitates testing whether $z+p\Z$ lies in $\SR_d(p)$
for $p\mid b$.
\end{itemize}
\end{remarks}

\begin{example}
For $k=33$ and $d=5$ we have $\CR(d)=\{2\}$ and $\sgn z=+1$.
For $\zmax=10^{16}$ this leaves $2\times 10^{15}$ candidate pairs $(5,z)$ to check.
We have $\#\AR_d(q)=14$ with $q=891$, which reduces this to approximately $3.143\times 10^{13}$ candidate pairs.
The table below shows the benefit of including additional primes $p\mid a$.
\smallskip

\begin{center}
\begin{tabular}{rrrrr}
$p\mid a$ & $\#\SR_d(p)$ & $\#\ZR(m)$ & $m$ & $\#\ZR(m,s,\zmax)$\\\toprule
 -- & -- & $14$ & $4455$ & $3.143\times 10^{13}$\\
$2$ & $1$ & $14$ & $8910$ & $1.571\times 10^{13}$\\
$7$ & $1$ & $14$ & $62370$ & $2.245\times 10^{12}$\\
$13$ & $3$ & $42$ & $810810$ & $5.180\times 10^{11}$\\
$17$ & $9$ & $378$ & $13783770$ & $2.742\times 10^{11}$\\
$23$ & $12$ & $4536$ & $317026710$ & $1.431\times 10^{11}$\\
$29$ & $15$ & $68040$ & $9193774590$ & $7.401\times 10^{10}$\\
$43$ & $19$ & $1292760$ & $395332307370$ & $3.270\times 10^{10}$\\
$67$ & $27$ & $34904520$ & $26487264593790$ & $1.318\times 10^{10}$\\
$103$ & $43$ & $1500894360$ & $2728188253160370$ & $5.501\times 10^{9\phantom{0}}$\\\bottomrule
\end{tabular}
\end{center}
\bigskip

The net gain is a factor of more than $363541$ over the na\"ive approach;
we gain a factor of about $63$ from cubic reciprocity and local constraints
mod $q$, and a factor of about $5712$ from the $p\mid a$.
In general, including auxiliary $p\mid a$ ensures that the number of $(d,z)$ we need to consider for small values of $d$ is a negligible proportion of the total computation. \end{example}

\begin{remark}
With CRT enumeration, we avoid the need to store the sets $\ZR(m)$, analogs of
which were explicitly constructed in \cite{Booker}.  This greatly reduces the memory
required when~$d$ is small. In this way, we no longer
rely on computations of integral points on the elliptic curve defined by
\eqref{eq:sdz} to rule out very small values of
$d$. Nevertheless, we note that one can improve the integral point search
carried out in \cite{Booker}, using a trick of Bremner \cite{Bremner} to
pass to a 3-isogenous curve. Using this approach we were able to
unconditionally rule out any solutions to \eqref{eq:threecubes} with
$d\le 100$ for the $k$ listed in \eqref{eq:openk}, and with $d\le 20,000$
assuming the GRH. It is thus now possible to certify under GRH
Bremner's heuristic search of the same region in 1995.
\end{remark}

\section{Heuristics}\label{sec:heuristics}
In this section we present a heuristic analysis of the distribution of
solutions to \eqref{eq:threecubes} for a fixed $k$.  We then use this to
optimize the choice of the ratio $R\coloneqq \zmax/\dmax$.

From \eqref{eq:realdensity} we see that on
$V=\{(x,y,z)\in\R^3:x^3+y^3+z^3=0,\,|x|\ge|y|\ge|z|\}$,
the proportion of the real density contributed by
points satisfying $y/z\in[t_1,t_2]$ is
\begin{equation}\label{eq:yzproportion}
4\sigma_\infty^{-1}\int_{t_1}^{t_2}\frac{dt}{(t^3+1)^{2/3}}.
\end{equation}
Given a large solution $(x,y,z)\in\Z^3$ to $x^3+y^3+z^3=k$, with
$|x|\ge|y|\ge|z|$, the
projective point $[x:y:z]\in\mathbb{P}^2(\R)$ lies close to the Fermat
curve $x^3+y^3+z^3=0$. We conjecture that for fixed $k$,
the ratios $y/z$ are distributed as above: the proportion of
points (ordered by any height function as in
\S\ref{sec:solutiondensity}) with $y/z\in[t_1,t_2]$ should converge to the quantity in
\eqref{eq:yzproportion}.

Let us assume that this is the case and work out the distribution of
$r\coloneqq -\frac{z}{x+y}$ for $(x,y,z)\in V$.
We have
$$
-\frac{y}{x}=\frac{2r^3+1-\sqrt{12r^3-3}}{2(r^3-1)}
\quad\text{and}\quad
-\frac{z}{x}=\frac{r(\sqrt{12r^3-3}-3)}{2(r^3-1)},
$$
so that
$$
t:=\frac{y}{z}=\frac{\sqrt{12r^3-3}-3}{6r}
\quad\text{and}\quad
(t^3+1)^{-2/3}\frac{dt}{dr}=\sqrt{\frac{3}{4r^3-1}}.
$$
Hence, for any $R\ge\alpha^{-1}$ we have
$$
\Pr[r\le R]=1-\Pr[r>R]=1-4\sigma_\infty^{-1}\int_R^\infty
\sqrt{\frac{3}{4r^3-1}}\,dr
=1-cK(R),
$$
where $c=4\sqrt3\sigma_\infty^{-1}
=6\sqrt3\frac{\Gamma(2/3)}{\Gamma(1/3)^2}
=1.96084321968938583\ldots$
and
$$
K(R)\coloneqq \int_R^\infty\frac{dr}{\sqrt{4r^3-1}}
=R^{-1/2}\sum_{j=0}^\infty
\frac{\binom{j-\frac12}{j}}{1+6j}(4R^3)^{-j}.
$$

Thus, the values of $1-cK(r)$ should be uniformly distributed on $[0,1]$.
To test this hypothesis, we plotted the cumulative distribution
of $1-cK(-z/(x+y))$ over the points of the Huisman data set with
$10^{7.5}<|x|\le 10^{15}$ versus that of a uniform random variable; see
Figure~\ref{fig:uniform}.

\begin{figure}
\begin{tikzpicture}
\selectcolormodel{gray}
\begin{axis}[
width=6in,
height=4in,
enlargelimits=false,
]
% uncomment for final version
\addplot+[only marks, scatter, mark=x, mark size=0.01cm] table{zdplot.dat};
\addplot+[mark=x] coordinates {
	(0,0)
	(1,1)
};
\end{axis}
\end{tikzpicture}
\caption{Cumulative distribution of $1-cK(-z/(x+y))$ over solutions $(x,y,z)$
in the Huisman data set with $\max\{|x|,|y|,|z|\}\in [10^{7.5},10^{15}]$,
versus a uniform random variable.}\label{fig:uniform}
\end{figure}
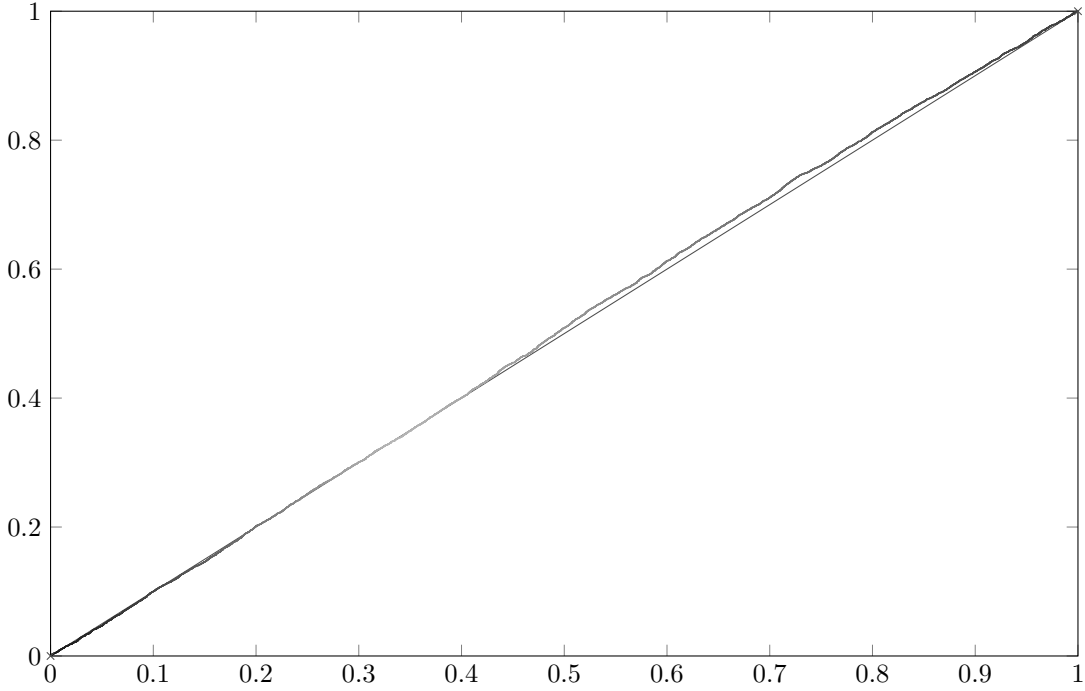

\begin{example}
For our solution to $x^3+y^3+z^3=3$ we have
$$
r\approx 4.36\times10^6
\quad\text{and}\quad
cK(r)\approx 9.39\times10^{-4},
$$
so this solution was an approximately 1-in-1000 event. This is also
reflected by the fact that the solution is highly skewed, with $|x|$ and
$|y|$ both much larger than $|z|$.
\end{example}

We use this analysis to optimize the choice of $R=\zmax/\dmax$ as follows.
We assume that a given divisor $d\in\Z_{>0}$ occurs with probability
$\kappa_d/d$, where $\kappa_d$ is an arithmetic factor
(depending on $k$) encoding the local solubility, in such a way that
$$
\sum_{d\le x}\kappa_d=\rho x+O(x/\log^2{x}),
\quad\text{for some constant }\rho>0.
$$
By partial summation it follows that there exists $C$ such that
$$
\sum_{d\le x}\frac{\kappa_d}{d}=\rho\log{x}+C+o(1)
\quad\text{and}\quad
\sum_{d\le x}\kappa_d f(d)
=(\rho+o(1))\int_0^x f(u)\,du
\quad\text{as }x\to\infty,
$$
for any monotonically decreasing function $f$ satisfying
$f(u)\asymp u^{-s}$ for some $s\in(0,1)$.
In turn, we expect to find $z$ in a fixed arithmetic progression modulo $d\le\dmax$ with
probability
$$
\Pr[|z|\le\zmax\mid d\text{ fixed}]=\Pr[r\le\zmax/d]=1-cK(\zmax/d).
$$
Hence, the number of
solutions that we expect to find is
\begin{align*}
\sum_{d\le\dmax}\frac{\kappa_d}{d}\left(1-cK\!\left(\frac{\zmax}{d}\right)\right)
&=\rho\log\dmax+C+o(1)
-(\rho+o(1))c\int_0^{\dmax}K\!\left(\frac{\zmax}{u}\right)\frac{du}{u}\\
&=\rho\log\dmax+C-\rho c\int_{\zmax/\dmax}^\infty K(r)\frac{dr}{r}+o(1).
\end{align*}
Taking $\dmax=\alpha\zmax$ recovers Heath-Brown's conjecture, provided that
$\rho=\rhos$.

Next, suppose that the total running time is $T(\dmax,\zmax)$, and let
$T_d$ and $T_z$ denote its partial derivatives.
Let $\dmax$ be defined implicitly in terms of $R=\zmax/\dmax$ so that $(\dmax,\zmax)$ remains on a level set for $T$, meaning that
$$
T(\dmax,R\dmax)=\text{constant}.
$$
Differentiating with respect to $R$, we have
$$
T_d(\dmax,R\dmax)\frac{\partial\dmax}{\partial R}
+T_z(\dmax,R\dmax)\left(\dmax+R\frac{\partial\dmax}{\partial R}\right)=0.
$$
We seek to maximize the expected solution count, which to leading order
is
$$
\rho\log\dmax+C-\rho c\int_R^\infty K(r)\frac{dr}{r}.
$$
Differentiating with respect to $R$, this gives
$$
\frac{\rho}{\dmax}\frac{\partial\dmax}{\partial R}
+\frac{\rho cK(R)}{R}=0,
$$
so that $\frac{\partial\dmax}{\partial R}=-c\dmax K(R)/R$.
Substituting this into the above, we obtain
$$
\frac{T_d(\dmax,R\dmax)}{T_z(\dmax,R\dmax)}
=R\left(\frac1{cK(R)}-1\right)
\approx
c^{-1}R^{3/2}\left(1-\frac1{56R^3}\right)-R =: C_R.
$$

In Table~\ref{table:tdtz} we show computed $T_d/T_z$ ratios for $k=3$ and various values of $R$ and $\dmax$.
For a given $\dmax$ we wish to choose~$R$ so that $T_d/T_z\approx C_R$.  It is difficult to measure $T_d/T_z$ precisely; it is the ratio of two small numbers, and this ratio is easily influenced by small differences in timings when running computations on different hardware.  To compute the values below we used a single hardware platform and took medians of five runs to compute each row.

From Table~\ref{table:tdtz} we can see that for $k=3$ and $\dmax \ge 2^{35}$ the optimal choice of $R$ is greater than 32, and for $\dmax \ge 2^{50}$ it is greater than 64.  
For other values of $k$ the pattern is similar although the $T_d/T_z$ vary slightly; this is to be expected given the varying benefit of cubic reciprocity constraints.

\begin{table}[tbh!]
\begin{tabular}{rrrrccr}
$R$ & $\dmax$ & $\zmax$ & $T_d$ & $T_z$ & $T_d/T_z$ & $C_R$\\\toprule
 $32$ & $2^{35}$ & $2^{40}$ & $2.804\times 10^{-08}$ & $2.359\times 10^{-10}$ & $118.9$ &  $60.3$\\
 $32$ & $2^{40}$ & $2^{45}$ & $2.738\times 10^{-08}$ & $2.247\times 10^{-10}$ & $121.8$ &  $60.3$\\
 $32$ & $2^{45}$ & $2^{50}$ & $2.922\times 10^{-08}$ & $2.175\times 10^{-10}$ & $134.4$ &  $60.3$\\
 $32$ & $2^{50}$ & $2^{55}$ & $3.113\times 10^{-08}$ & $2.150\times 10^{-10}$ & $144.8$ &  $60.3$\\
 $32$ & $2^{55}$ & $2^{60}$ & $3.678\times 10^{-08}$ & $2.100\times 10^{-10}$ & $175.2$ &  $60.3$\\
 $64$ & $2^{35}$ & $2^{41}$ & $3.140\times 10^{-08}$ & $1.813\times 10^{-10}$ & $173.2$ & $197.1$\\
 $64$ & $2^{40}$ & $2^{46}$ & $2.771\times 10^{-08}$ & $1.730\times 10^{-10}$ & $160.2$ & $197.1$\\
 $64$ & $2^{45}$ & $2^{51}$ & $3.112\times 10^{-08}$ & $1.613\times 10^{-10}$ & $192.9$ & $197.1$\\
 $64$ & $2^{50}$ & $2^{56}$ & $3.187\times 10^{-08}$ & $1.506\times 10^{-10}$ & $211.6$ & $197.1$\\
 $64$ & $2^{55}$ & $2^{61}$ & $3.862\times 10^{-08}$ & $1.612\times 10^{-10}$ & $239.6$ & $197.1$\\
$128$ & $2^{35}$ & $2^{42}$ & $3.749\times 10^{-08}$ & $1.238\times 10^{-10}$ & $302.8$ & $618.5$\\
$128$ & $2^{40}$ & $2^{47}$ & $3.407\times 10^{-08}$ & $1.216\times 10^{-10}$ & $280.2$ & $618.5$\\
$128$ & $2^{45}$ & $2^{52}$ & $3.826\times 10^{-08}$ & $1.530\times 10^{-10}$ & $250.1$ & $618.5$\\
$128$ & $2^{50}$ & $2^{57}$ & $3.768\times 10^{-08}$ & $1.185\times 10^{-10}$ & $318.0$ & $618.5$\\
$128$ & $2^{55}$ & $2^{62}$ & $4.096\times 10^{-08}$ & $1.091\times 10^{-10}$ & $375.4$ & $618.5$\\\bottomrule
\end{tabular}
\bigskip

\caption{$T_d/T_z$ vs $C_R$ for various values of $\dmax$ and $R=\zmax/\dmax$ for $k=3$.}\label{table:tdtz}
\end{table}

\section{Computational results}\label{sec:computation}
\subsection{Implementation}
We implemented the algorithm described in Section~\ref{sec:algorithm}
using the \texttt{gcc} C compiler \cite{GCC} and the \texttt{primesieve}
library for fast prime enumeration \cite{primesieve}.  We parallelized
by partitioning the set of primes $p\le \dmax$ into sub-intervals
$[\pmin,\pmax]$ of suitable size, with the work distributed across jobs
that checked all the $(d,z)$ candidates with the largest prime
factor $p_1\mid d$ lying in the assigned interval.
Each job was run
on a separate machine, with local parallelism achieved by distributing
the~$p_1$ across available cores (and for small values of $p_1$ also
distributing the $p_2$), as noted in Remarks~\ref{rem:parallel}.
When choosing the number of jobs and the sizes of the intervals $[\pmin,\pmax]$
we use the $\rhoz$ density estimates derived in Section~\ref{sec:solutiondensity},
as noted in Remark~\ref{rem:rhoap}.

We used a standard Tonelli--Shanks approach to computing cube roots
modulo primes; this involves computing a discrete logarithm in the
$3$-Sylow subgroup of $(\Z/p\Z)^\times$, using $O(1)$ group operations on
average, and $O(1)$ exponentiations. Hensel lifting was used to compute
cube roots modulo prime powers; these were precomputed and cached
for all prime powers up to $\min\{\pmax,\sqrt{\dmax}\}$.
For the the values of $\dmax$ that
we used, this precomputation typically takes just a few seconds and the
cache size is well under one gigabyte.  We use Montgomery representation
\cite{Montgomery} for performing arithmetic in $(\Z/p^r\Z)^\times$, but
switch to standard integer representation and use Barrett reduction
\cite{Barrett} during CRT enumeration of cube roots of $k$ modulo $d$,
and when sieving arithmetic progressions via auxiliary primes.

For the $k$ of interest, the sets $\AR_d(q)$ giving constraints modulo the integer~$q$ defined in
Lemma~\ref{lem:admissibility} for admissible pairs $(d,z)$ were
precomputed and cached; again this takes only a few seconds for the largest
values of $k$. In order to avoid using arithmetic progressions of
modulus larger than $\zmax$ we project these constraints to residue
classes modulo a suitably chosen divisor of $q$ when $qd>\zmax$.

\subsection{Computations}
In September 2019 we ran computations for the eleven unresolved $k\le
1000$ listed in \eqref{eq:openk} on Charity Engine's crowd-sourced
compute grid consisting of approximately 500,000 personal computers.
For this initial search we used $\zmax=10^{17}$ and $\dmax=\alpha\zmax$
to search for all solutions to \eqref{eq:threecubes} with $\min\{|x|,|y|,|z|\}\le 10^{17}$.  This search yielded the solutions for $k=42$, $k=165$, and $k=906$ listed in the introduction.
We then ran a search for $k=3$ using $\zmax=10^{18}$ and $\dmax=\alpha\zmax/9$ and found the solution for $k=3$ listed in the introduction.
These computations involved a total of several hundred core-years but were completed in just a few weeks (it is difficult to give more precise estimates of the computational costs due to variations in processor speeds and resource availability in a crowd-sourced computation).
Subsequently, over the course of 2020, Charity Engine conducted a search
at lower priority for the remaining eight candidate values of $k$,
with $\zmax=10^{19}$ and $\dmax=\zmax/54$;
this yielded the solution for $k=579$ in January 2021.

\begin{remark}
While in principle these searches rule out the existence of any solutions that were not found, we are reluctant to make any unconditional claims.  Despite putting in place measures to detect failures, including counting the primes that were enumerated (these counts can be efficiently verified after the fact), there is always the possibility of undetected hardware or software errors, especially on a large network of personal computers that typically do not have error correcting memory.
\end{remark}

In order to verify the minimality of the solution we found for $k=3$, we ran a separate verification with $\zmax$ equal to $472715493453327032$, the absolute value of the $z$ in our solution,  and $\dmax=\alpha\zmax$. This search was run on Google's Compute Engine \cite{GCE} and found no solutions other than those already known.
These computations were run on 8-core (16-vCPU) instances equipped with Intel Xeon processors in the Sandybridge, Haswell, and Broadwell families running at 2.0GHz or 2.2GHz.
Using 155,579 nodes the computation took less than 4 hours and used approximately 120 core-years.  We detected errors in 5 of the 155,579 runs which were corrected upon re-running the computations.  Barring the existence of any undetected errors, these computations rule out any smaller solutions for $k=3$ other than those we now know.

To assess the benefit of the theoretical and algorithmic improvements
introduced here, we searched for solutions to $k=33$ using $R=64$, which
is close to the optimal choice for $\dmax$ in the range
$[2^{40},2^{50}]$. The general search strategy we envision is to start
with a value of $\dmax$ for which all solutions with $|z|\le R\dmax$ are
known, where $R$ is chosen optimally for $\dmax$.  One would then
successively double $\dmax$, adjusting $R$ as necessary, and run a
search using $\zmax=R\dmax$.  If one takes care to avoid checking the
same admissible $(d,z)$ twice, the total time is approximately equal to
a single complete search using the final values of $\dmax$ and $R$ (one
expects $R$ to be increasing).  The first $\dmax=2^n$ sufficient to find
a solution for $k=33$ with this strategy is $\dmax=2^{47}$, for which we
choose $R=64$, yielding $\zmax=2^{53}$.  Using 2.8GHz Intel processors
in the Skylake family, this search finds the known solution for $k=33$
in 107 core-days.  The search in \cite{Booker} using $\zmax=10^{16}$ and
$\dmax=\alpha\zmax$ took 3145 core-days running mostly on 2.6GHz Intel processors in the Sandybridge family.  After adjusting for the difference in processor speeds and $\zmax$ values, our new approach finds the first solution for $k=33$ approximately 25 times faster.

In the future we hope to use this strategy to search for solutions for
the seven $k\le 1000$ that remain unresolved:
\begin{equation}
114,\ 390,\ 627,\ 633,\ 732,\ 921,\ 975.
\end{equation}

An implementation of our algorithm is available at
\begin{center}
\url{https://github.com/AndrewVSutherland/SumsOfThreeCubes}.
\end{center}

\bibliographystyle{amsalpha}
\providecommand{\href}[2]{#2}

\end{document}